\documentclass[12pt,reqno]{article}
\usepackage[usenames]{color}
\usepackage[colorlinks=true,linkcolor=webgreen,filecolor=webbrown,citecolor=webgreen]{hyperref}

\usepackage{amssymb}
\usepackage{amsmath}
\usepackage{amsfonts}

\usepackage[T1]{fontenc}

\usepackage{enumerate}

\definecolor{webgreen}{rgb}{0,.5,0}
\definecolor{webbrown}{rgb}{.6,0,0}
\newtheorem{theorem}{Theorem}

\newtheorem{lemma}[theorem]{Lemma}

\newtheorem{proposition}[theorem]{Proposition}

\newenvironment{proof}[1][Proof]{\noindent\textbf{#1.} }{\ \rule{0.5em}{0.5em}}

\allowdisplaybreaks

\begin{document}

\begin{center}
\vskip1cm

{\LARGE \textbf{A problem of I. Ra\c{s}a on Bernstein polynomials and convex functions}}

\vspace{2cm}

{\large Ulrich Abel}\\[3mm]
\textit{Fachbereich MND}\\[0pt]
\textit{Technische Hochschule Mittelhessen}\\[0pt]
\textit{Wilhelm-Leuschner-Stra\ss e 13, 61169 Friedberg, }\\[0pt]
\textit{Germany}\\[0pt]
\href{mailto:Ulrich.Abel@mnd.thm.de}{\texttt{Ulrich.Abel@mnd.thm.de}}
\end{center}

\vspace{2cm}

{\large \textbf{Abstract.}}

\bigskip

We present an elementary proof of a conjecture by I. Ra\c{s}a which is an inequality involving Bernstein basis polynomials and convex functions. It was affirmed in positive very recently by the use of stochastic convex orderings. Moreover, we derive the corresponding results for Mirakyan-Favard-Sz\'{a}sz operators and Baskakov operators.

\bigskip



\vspace{2cm}

\section{Introduction}  

For $n=0,1,2,\ldots $ and $\nu =0,1,\ldots ,n$, let 
\begin{equation*}
p_{n,\nu }\left( x\right) =\binom{n}{\nu }x^{\nu }\left( 1-x\right) ^{n-\nu }
\end{equation*}%
denote the Bernstein basis polynomials. For $\nu >n$ we define $p_{n,\nu
}\left( x\right) =0$. 
Very recently, J. Mrowiec, T. Rajba and S. W\k{a}sowicz \cite{Mrowiec-ea-2016}\ proved the following theorem.

\begin{theorem}[\protect\cite{Mrowiec-ea-2016}]
\label{theorem-conj-bernstein}Let $n\in \mathbb{N}$. If $f\in C\left[ 0,1%
\right] $ is a convex function, then 
\begin{equation}
\sum\limits_{i=0}^{n}\sum\limits_{j=0}^{n}\left[ p_{n,i}\left( x\right)
p_{n,j}\left( x\right) +p_{n,i}\left( y\right) p_{n,j}\left( y\right)
-2p_{n,i}\left( x\right) p_{n,j}\left( y\right) \right] {f}\left( \frac{i+j}{%
2n}\right) \geq 0,  \label{conjecture-Bernstein}
\end{equation}%
for all $x,y\in \left[ 0,1\right] $.
\end{theorem}

This inequality involving Bernstein basis polynomials and convex functions was stated as an open problem 25 years ago by I. Ra\c{s}a. J. Mrowiec, T. Rajba and S. W\k{a}sowicz \cite{Mrowiec-ea-2016}\ affirmed the conjecture in positive. Their proof makes heavy use of probability theory. As a tool they applied stochastic convex orderings (which they proved for binomial distributions) as well as the so-called concentration inequality.

The purpose of this note is a short elementary proof of the above theorem.
It even doesn't take advantage of the Levin--Ste\v{c}kin theorem \cite{Levin-1960} (see also \cite{Niculescu-Persson-book-2006}, Theorem 4.2.7) which gives necessary and sufficient conditions on $F$ for the
non-negativity of $\int_{a}^{b}f\left( x\right) dF\left( x\right) $ if $f$ is a convex function on $\left[ a,b\right] $.

In the last two sections we prove the corresponding results for Mirakyan-Favard-Sz\'{a}sz operators and Baskakov operators.

\section{An elementary proof of Theorem~\protect\ref{theorem-conj-bernstein}}

We start with a few auxiliary results.

\begin{lemma}
\label{lemma-bernstein1}For $n,k\in \mathbb{N}$, 
\begin{equation*}
\sum\limits_{i=0}^{k}p_{n,i}\left( x\right) p_{n,k-i}\left( y\right) =\frac{%
1}{k!}\left[ \left( \frac{\partial }{\partial z}\right) ^{k}\left(
1+xz\right) ^{n}\left( 1+yz\right) ^{n}\right] _{\mid z=-1}.
\end{equation*}
\end{lemma}

Note that the formula is valid also if $k>n$.

\begin{proof}
We have 
\begin{eqnarray*}
&&\sum\limits_{i=0}^{k}p_{n,i}\left( x\right) p_{n,k-i}\left( y\right) \\
&=&\sum\limits_{i=0}^{k}\binom{n}{i}x^{i}\left( 1-x\right) ^{n-i}\binom{n}{%
k-i}y^{k-i}\left( 1-y\right) ^{n-\left( k-i\right) } \\
&=&\frac{1}{k!}\left\{ \sum\limits_{i=0}^{k}\binom{k}{i}\left[ \left( \frac{%
\partial }{\partial z}\right) ^{i}\left( 1+xz\right) ^{n}\right] \left[
\left( \frac{\partial }{\partial z}\right) ^{k-i}\left( 1+yz\right) ^{n}%
\right] \right\} _{\mid z=-1}
\end{eqnarray*}%
and the lemma follows by an application of the Leibniz rule for the
differentiation of products of functions.
\end{proof}

The next result is a representation of the left-hand side of Eq. $\left( \ref%
{conjecture-Bernstein}\right) $.

\begin{lemma}
\label{lemma-bernstein2}%
\begin{eqnarray*}
&&\sum\limits_{i=0}^{n}\sum\limits_{j=0}^{n}\left[ p_{n,i}\left( x\right)
p_{n,j}\left( x\right) +p_{n,i}\left( y\right) p_{n,j}\left( y\right)
-2p_{n,i}\left( x\right) p_{n,j}\left( y\right) \right] {f}\left( \frac{i+j}{%
2n}\right) \\
&=&\sum\limits_{k=0}^{2n}{f}\left( \frac{k}{2n}\right) \frac{1}{k!}\left. %
\left[ \left( \frac{\partial }{\partial z}\right) ^{k}\left[ \left(
1+xz\right) ^{n}-\left( 1+yz\right) ^{n}\right] ^{2}\right] \right\vert
_{z=-1}.
\end{eqnarray*}
\end{lemma}

\begin{proof}
It is a direct consequence of the preceding lemma that 
\begin{eqnarray*}
&&\sum\limits_{i=0}^{n}\sum\limits_{j=0}^{n}\left[ p_{n,i}\left( x\right)
p_{n,j}\left( x\right) +p_{n,i}\left( y\right) p_{n,j}\left( y\right)
-2p_{n,i}\left( x\right) p_{n,j}\left( y\right) \right] {f}\left( \frac{i+j}{%
2n}\right) \\
&=&\sum\limits_{k=0}^{2n}{f}\left( \frac{k}{2n}\right) \frac{1}{k!}\left.
\left( \frac{\partial }{\partial z}\right) ^{k}\left[ \left( 1+xz\right)
^{2n}+\left( 1+yz\right) ^{2n}-2\left( 1+xz\right) ^{n}\left( 1+yz\right)
^{n}\right] \right\vert _{z=-1}
\end{eqnarray*}%
and the lemma follows, by the binomial formula.
\end{proof}

For fixed $n\in \mathbb{N}$ and $x,y\in \left[ 0,1\right] $, we define 
\begin{equation*}
g\left( z\right) \equiv g_{n}\left( z;x,y\right) =\left( \frac{\left(
1+xz\right) ^{n}-\left( 1+yz\right) ^{n}}{z}\right) ^{2}.
\end{equation*}%
Note that $g$ is a polynomial in $z$ of degree at most $2n-2$.

The next proposition is the key result.

\begin{proposition}
\label{prop-bernstein}Let $\left( a_{k}\right) _{k=0}^{2n}$ be a real
sequence and fix $x,y\in \left[ 0,1\right] $. Then, 
\begin{equation}
\sum\limits_{i=0}^{n}\sum\limits_{j=0}^{n}\left[ p_{n,i}\left( x\right)
p_{n,j}\left( x\right) +p_{n,i}\left( y\right) p_{n,j}\left( y\right)
-2p_{n,i}\left( x\right) p_{n,j}\left( y\right) \right] \cdot
a_{i+j}=\sum\limits_{k=0}^{2n-2}\left( \Delta ^{2}a_{k}\right) \frac{1}{k!}%
g^{\left( k\right) }\left( -1\right)  \label{bernstein-identity}
\end{equation}%
and $g^{\left( k\right) }\left( -1\right) \geq 0$, for $k=0,1,\ldots ,2n-2$.
\end{proposition}

Here $\Delta $ denotes the forward difference $\Delta a_{k}:=a_{k+1}-a_{k}$
such that $\Delta ^{2}a_{k}=a_{k+2}-2a_{k+1}+a_{k}$.

Because $g$ is a polynomial in $z$ of degree at most $2n-2$, it is obvious
that $g^{\left( 2n-1\right) }\left( -1\right) =g^{\left( 2n\right) }\left(
-1\right) =0$.

\begin{proof}
Observe that 
\begin{equation*}
\left( z^{2}g\left( z\right) \right) ^{\left( k\right) }=z^{2}g^{\left(
k\right) }\left( z\right) +\binom{k}{1}\cdot 2zg^{\left( k-1\right) }\left(
z\right) +\binom{k}{2}\cdot 2g^{\left( k-2\right) }\left( z\right) .
\end{equation*}%
By Lemma~\ref{lemma-bernstein2}, we have 
\begin{eqnarray*}
&&\sum\limits_{i=0}^{n}\sum\limits_{j=0}^{n}\left[ p_{n,i}\left( x\right)
p_{n,j}\left( x\right) +p_{n,i}\left( y\right) p_{n,j}\left( y\right)
-2p_{n,i}\left( x\right) p_{n,j}\left( y\right) \right] \cdot a_{i+j} \\
&=&\sum\limits_{k=0}^{2n}a_{k}\frac{1}{k!}\left. \left( \frac{\partial }{%
\partial z}\right) ^{k}\left( z^{2}g\left( z\right) \right) \right\vert
_{z=-1} \\
&=&\sum\limits_{k=0}^{2n-2}a_{k}\frac{1}{k!}g^{\left( k\right) }\left(
-1\right) -2\sum\limits_{k=1}^{2n-1}a_{k}\frac{1}{\left( k-1\right) !}%
g^{\left( k-1\right) }\left( -1\right) +\sum\limits_{k=2}^{2n}a_{k}\frac{1}{%
\left( k-2\right) !}g^{\left( k-2\right) }\left( -1\right) \\
&=&\sum\limits_{k=0}^{2n-2}\left( a_{k}-2a_{k+1}+a_{k+2}\right) \frac{1}{k!}%
g^{\left( k\right) }\left( -1\right)
\end{eqnarray*}%
which proves Eq. $\left( \ref{bernstein-identity}\right) $. Noting that 
\begin{eqnarray*}
g\left( z\right) &=&\left( x-y\right) ^{2}\left( \frac{\left( 1+xz\right)
^{n}-\left( 1+yz\right) ^{n}}{\left( 1+xz\right) -\left( 1+yz\right) }%
\right) ^{2} \\
&=&\left( x-y\right) ^{2}\left( \sum\limits_{k=0}^{n-1}\left( 1+xz\right)
^{k}\left( 1+yz\right) ^{n-1-k}\right) ^{2}
\end{eqnarray*}%
it is immediate that $g^{\left( k\right) }\left( -1\right) \geq 0$, for $%
k=0,1,\ldots ,2n-2$, if $x,y\in \left[ 0,1\right] $.
\end{proof}

Using the elementary formula 
\begin{equation*}
\left( \frac{b^{n}-a^{n}}{b-a}\right)
^{2}=\sum\limits_{j=0}^{n-1}a^{j}b^{2n-2-j}\min \left\{ j,2n-2-j\right\}
\end{equation*}%
we obtain more precisely 
\begin{equation*}
g\left( z\right) =\left( x-y\right) ^{2}\sum\limits_{j=0}^{2n-2}\left(
1+xz\right) ^{j}\left( 1+yz\right) ^{2n-2-j}\min \left\{ j,2n-2-j\right\} .
\end{equation*}%
Hence, 
\begin{eqnarray*}
g^{\left( k\right) }\left( -1\right) &=&k!\left( x-y\right)
^{2}\sum\limits_{j=0}^{2n-2}\min \left\{ j,2n-2-j\right\} \\
&&\times \sum\limits_{i=0}^{k}\binom{j}{i}x^{i}\left( 1-x\right) ^{j-i}%
\binom{2n-2-j}{k-i}y^{k-i}\left( 1-y\right) ^{2n-2-j-\left( k-i\right) } \\
&=&k!\left( x-y\right) ^{2}\sum\limits_{j=0}^{2n-2}\min \left\{
j,2n-2-j\right\} \sum\limits_{i=0}^{k}p_{j,i}\left( x\right)
p_{2n-2-j,k-i}\left( y\right) .
\end{eqnarray*}

\begin{proof}[Proof of Theorem~\protect\ref{theorem-conj-bernstein}]
For $k=0,1,\ldots ,2n-2$, we put 
\begin{equation*}
a_{k}={f}\left( \frac{k}{2n}\right) .
\end{equation*}%
If $f\in C\left[ 0,1\right] $ is a convex function it follows that $\Delta
^{2}a_{k}\geq 0$, for $k=0,1,\ldots ,2n-2$. Therefore, application of
Proposition~\ref{prop-bernstein} proves Theorem~\ref{theorem-conj-bernstein}.
\end{proof}

\section{Mirakyan-Favard-Sz\'{a}sz operators}

The Mirakyan-Favard-Sz\'{a}sz $S_{n}$ operators associate to each function $%
f $ of (at most) exponential growth on $\left[ 0,\infty \right) $ the
function 
\begin{equation*}
\left( S_{n}f\right) \left( x\right) =e^{-nx}\sum\limits_{\nu =0}^{\infty }%
\frac{\left( nx\right) ^{\nu }}{\nu !}{f}\left( \frac{\nu }{n}\right) \text{
\qquad }\left( x\in \left[ 0,\infty \right) \right) .
\end{equation*}%
If, for $n,\nu =0,1,2,\ldots $, 
\begin{equation*}
s_{\nu }\left( x\right) =e^{-x}\frac{x^{\nu }}{\nu !}
\end{equation*}%
denote the corresponding basis functions, the operators can be written in
the form 
\begin{equation*}
\left( S_{n}f\right) \left( x\right) =\sum\limits_{\nu =0}^{\infty }s_{\nu
}\left( nx\right) {f}\left( \frac{\nu }{n}\right) \text{ \qquad }\left( x\in %
\left[ 0,\infty \right) \right) .
\end{equation*}

\begin{theorem}
\label{theorem-conj-mfs}Let $n\in \mathbb{N}$. If $f\in C\left[ 0,\infty
\right) $ is a convex function of (at most) exponential growth, then 
\begin{equation*}
\sum\limits_{i=0}^{\infty }\sum\limits_{j=0}^{\infty }\left[ s_{i}\left(
x\right) s_{j}\left( x\right) +s_{i}\left( y\right) s_{j}\left( y\right)
-2s_{i}\left( x\right) s_{j}\left( y\right) \right] {f}\left( i+j\right)
\geq 0,
\end{equation*}%
for all $x,y\in \left[ 0,\infty \right) $.
\end{theorem}

\begin{proof}
It can easily be verified that 
\begin{equation*}
\sum\limits_{i=0}^{k}s_{i}\left( x\right) s_{k-i}\left( y\right) =\frac{1}{%
k!}\left( x+y\right) ^{k}e^{-\left( x+y\right) }=\frac{1}{k!}\left[ \left( 
\frac{\partial }{\partial z}\right) ^{k}e^{\left( x+y\right) z}\right]
_{\mid z=-1}.
\end{equation*}%
Hence, 
\begin{eqnarray*}
&&\sum\limits_{i=0}^{\infty }\sum\limits_{j=0}^{\infty }\left[ s_{i}\left(
x\right) s_{j}\left( x\right) +s_{i}\left( y\right) s_{j}\left( y\right)
-2s_{i}\left( x\right) s_{j}\left( y\right) \right] {f}\left( i+j\right) \\
&=&\sum\limits_{k=0}^{\infty }\frac{1}{k!}{f}\left( k\right) \left. \left[
\left( \frac{\partial }{\partial z}\right) ^{k}\left(
e^{2xz}+e^{2yz}-2e^{\left( x+y\right) z}\right) \right] \right\vert _{z=-1}
\\
&=&\sum\limits_{k=0}^{\infty }\frac{1}{k!}{f}\left( k\right) \left. \left[
\left( \frac{\partial }{\partial z}\right) ^{k}\left( e^{xz}-e^{yz}\right)
^{2}\right] \right\vert _{z=-1}.
\end{eqnarray*}%
Now we put, for fixed $x,y\geq 0$, 
\begin{equation*}
g\left( z\right) =\left( \frac{e^{xz}-e^{yz}}{z}\right) ^{2}.
\end{equation*}%
Observe that 
\begin{equation*}
g\left( z\right) =\int_{x}^{y}\int_{x}^{y}e^{\left( u+v\right)
z}dudv=\sum\limits_{\nu =0}^{\infty }\frac{\left( z+1\right) ^{\nu }}{\nu !}%
\int_{x}^{y}\int_{x}^{y}\left( u+v\right) ^{\nu }e^{-\left( u+v\right) }dudv
\end{equation*}%
which implies that $g^{\left( k\right) }\left( -1\right) \geq 0$, for $%
k=0,1,\ldots $, if $x,y\geq 0$. As in the Bernstein case we conclude that 
\begin{eqnarray*}
&&\sum\limits_{i=0}^{\infty }\sum\limits_{j=0}^{\infty }\left[ s_{i}\left(
x\right) s_{j}\left( x\right) +s_{i}\left( y\right) s_{j}\left( y\right)
-2s_{i}\left( x\right) s_{j}\left( y\right) \right] {f}\left( i+j\right) \\
&=&\sum\limits_{k=0}^{\infty }{f}\left( k\right) \frac{1}{k!}\left. \left( 
\frac{\partial }{\partial z}\right) ^{k}\left( z^{2}g\left( z\right) \right)
\right\vert _{z=-1} \\
&=&\sum\limits_{k=0}^{\infty }\Delta ^{2}{f}\left( k\right) \cdot \frac{1}{%
k!}g^{\left( k\right) }\left( -1\right)
\end{eqnarray*}%
which completes the proof.
\end{proof}

\section{Baskakov operators}

The Baskakov operators $V_{n}$ associate to each function $f$ of polynomial
growth on $\left[ 0,\infty \right) $ the function 
\begin{equation*}
\left( V_{n}f\right) \left( x\right) =\sum\limits_{\nu =0}^{\infty
}b_{n,\nu }\left( x\right) {f}\left( \frac{\nu }{n}\right) \text{ \qquad }%
\left( x\in \left[ 0,\infty \right) \right) ,
\end{equation*}%
where 
\begin{equation*}
b_{n,\nu }\left( x\right) =\binom{n+\nu -1}{\nu }\frac{x^{\nu }}{\left(
1+x\right) ^{n+\nu }}
\end{equation*}%
denote the Baskakov basis functions. We have 
\begin{equation*}
\sum\limits_{i=0}^{k}b_{n,i}\left( x\right) b_{n,k-i}\left( y\right) =\frac{%
1}{k!}\left[ \left( \frac{\partial }{\partial z}\right) ^{k}\left(
1-xz\right) ^{-n}\left( 1-yz\right) ^{-n}\right] _{\mid z=-1}
\end{equation*}%
and 
\begin{eqnarray*}
&&\sum\limits_{i=0}^{\infty }\sum\limits_{j=0}^{\infty }b_{n,i}\left(
x\right) b_{n,j}\left( x\right) \left[ b_{n,i}\left( x\right) b_{n,j}\left(
x\right) +b_{n,i}\left( y\right) b_{n,j}\left( y\right) -2b_{n,i}\left(
x\right) b_{n,j}\left( y\right) \right] {f}\left( i+j\right) \\
&=&\sum\limits_{k=0}^{\infty }{f}\left( k\right) \frac{1}{k!}\left. \left[
\left( \frac{\partial }{\partial z}\right) ^{k}\left( \left( 1-xz\right)
^{-n}-\left( 1+yz\right) ^{-n}\right) ^{2}\right] \right\vert _{z=-1}
\end{eqnarray*}%
In a similar manner as in the Bernstein case one can show the following
theorem.

\begin{theorem}
\label{theorem-conj-baskakov}Let $n\in \mathbb{N}$. If $f\in C\left[
0,\infty \right) $ is a convex function of polynomial growth, then 
\begin{equation*}
\sum\limits_{i=0}^{\infty }\sum\limits_{j=0}^{\infty }\left[ b_{n,i}\left(
x\right) b_{n,j}\left( x\right) +b_{n,i}\left( y\right) b_{n,j}\left(
y\right) -2b_{n,i}\left( x\right) b_{n,j}\left( y\right) \right] {f}\left(
i+j\right) \geq 0,
\end{equation*}%
for all $x,y\in \left[ 0,\infty \right) $.
\end{theorem}

\strut

\thispagestyle{empty}

\end{document}